\documentclass[12pt,a4paper,oneside]{amsart}
\usepackage{color,latexsym,textcomp,lastpage,fancyhdr,calc,graphicx,asymptote,txfonts}
\usepackage{amsmath,amstext,amsthm,amssymb,amsxtra,amsfonts}
\usepackage[left=2.5cm,right=2.5cm,bottom=2cm,top=2cm]{geometry} 
\usepackage{caption}
\usepackage{subcaption}
\DeclareCaptionSubType[arabic]{figure}

\graphicspath{{Figures/}}

\newtheorem{theorem}{Theorem}[section]
\newtheorem*{theorem*}{Theorem}
\newtheorem*{main*}{Main Theorem}
\newtheorem{claim}{Claim}[section]
\newtheorem{corollary}{Corollary}[section]
\newtheorem{lemma}{Lemma}[section]

\newtheorem{remark}{Remark}[section]

\newtheorem{definition}{Definition}[section]
\theoremstyle{definition}

\newcommand{\beql}[1]{\begin{equation}\label{#1}}
\newcommand{\eeq}{\end{equation}}
\newcommand{\Ds}{\displaystyle}

\newcommand{\Abs}[1]{{\left|{#1}\right|}}

\newcommand{\Set}[1]{{\left\{{#1}\right\}}}

\newcommand{\ba}{\mathbf{a}}
\newcommand{\bb}{\mathbf{b}}
\newcommand{\bc}{\mathbf{c}}
\newcommand{\be}{\mathbf{e}}
\newcommand{\Bf}{\mathbf{f}}
\newcommand{\bg}{\mathbf{g}}
\newcommand{\bh}{\mathbf{h}}
\newcommand{\bu}{\mathbf{u}}
\newcommand{\bv}{\mathbf{v}}
\newcommand{\bw}{\mathbf{w}}
\newcommand{\bx}{\mathbf{x}}

\newcommand{\bZ}{{\mathsf{Z}}}

\newcommand{\btau}{{\vec{\tau}}}

\newcommand{\bbH}{{\mathbb{H}}}

\newcommand{\R}{{\mathbb{R}}}
\newcommand{\Z}{{\mathbb{Z}}}

\newcommand{\one}{{\bf 1}}

\newcommand{\inner}[2]{{\langle #1, #2 \rangle}}

\newcommand{\dens}{{\rm dens\,}}
\newcommand{\supp}{{\rm supp\,}}

\newcommand{\vol}{{\rm vol\,}}
\newcommand{\ft}[1]{\widehat{#1}}

\newcommand{\fzero}[1]{{{\mathsf Z}\left({\ft{{#1}}}\right)}}
\renewcommand{\vec}[1]{{\mbox{\boldmath$#1$}}}

\newcommand{\diam}{{\rm diam\,}}

\newcommand{\bxi}{{\vec{\xi}}}

\newcounter{rem}
\newcounter{step}
\newcounter{mysec}
\newcounter{mysubsec}[mysec]
\newcounter{othm}
\def\theothm{\Alph{othm}}

\newcounter{fcap}
\begin{document}

\title{Structure results for multiple tilings in 3D}

\author[N.\,Gravin]{Nick Gravin}

\author[M.\,Kolountzakis]{Mihail N. Kolountzakis}
\address{M. \,K.: Department of Mathematics, University of Crete, Knossos Ave., GR-714 09, Iraklio, Greece}
\email{kolount@math.uoc.gr}
\thanks{M. \,K.: Supported by research grant No 3223 from the Univ.\ of Crete and in part by the Singapore MOE Tier 2 research grant MOE2011-T2-1-090}

\author[S.\,Robins]{Sinai Robins}
\thanks{N.\,G., S.\,R., D.\,S.:  Supported in part by the Singapore MOE Tier 2 research grant MOE2011-T2-1-090}
\author[D.\,Shiryaev]{Dmitry Shiryaev}
\address{N.\,G., S.\,R., D.\,S.: Division of Mathematical Sciences, Nanyang Technological University \\
SPMS, MAS-03-01, 21 Nanyang Link, Singapore 637371}
\email{ngravin@pmail.ntu.edu.sg, rsinai@ntu.edu.sg, shir0010@ntu.edu.sg}
\date{\today}


\maketitle


\section{Introduction}

The study of multiple tilings of Euclidean space began in 1936,
when the famous Minkowski facet-to-facet conjecture \cite{minkowski1907diophantische}
for classical tilings was extended to the setting of
$k$-tilings  with the unit cube, by Furtw\"angler~\cite{Furtwangler}.
Minkowski's facet-to-facet conjecture states that for any {\em lattice} tiling of $\R^d$
by translations of the unit cube, there exist at least two translated cubes that share a facet
(face of co-dimension 1).
This conjecture was strengthened by Keller \cite{keller-conjecture}
who conjectured the same conclusion for {\em any} cube tiling,
not just lattice tilings.
It was also strengthened in a different direction by Furtw\"angler~\cite{Furtwangler} who, again, conjectured the
same conclusion for any {\em multiple} lattice tiling.

To define a multiple tiling, suppose we translate a convex body $P$ with a discrete multiset $\Lambda$,
in such a way that each point of $\R^d$ gets covered exactly $k$ times,
except perhaps the translated copies of the boundary of $P$.
We then call such a body a $k$-tiler, and such an action has been given
the following names in the literature:
a  {\bf $k$-tiling},  a  {\bf tiling at level $k$},  a  {\bf tiling with multiplicity $k$},
and sometimes simply a {\bf multiple tiling}.
We may use any of these synonyms here,
and we immediately point out, for polytopes $P$,
a trivial but useful algebraic equivalence for a tiling at level $k$:
\begin{equation}\label{k.tiling.def}
\sum_{\lambda \in \Lambda}   \one_{P+\lambda}(x) = k,
\end{equation}
for almost all $x \in \R^d$, where $\one_{P}$ is the indicator function of the polytope $P$.

Furtw\"angler's conjecture was disproved by Haj\'os \cite{hajos1938} for dimension larger than 3 and
for $k \ge 9$
while Furtw\"angler himself \cite{Furtwangler} proved it for dimension at most 3.
Haj\'os \cite{Hajos} also proved Minkowski's conjecture in all dimensions.
The ideas of Furtw\"angler were subsequently
extended (but still restricted to cubes) by the important work of
Perron~\cite{perron1940},
Robinson~\cite{Robinson}, Szab\'o \cite{szabo1982multiple}, Gordon~\cite{Gordon}
and Lagarias and Shor~\cite{lagarias1992keller}.
These authors showed that for some levels $k$ and dimensions $d$ and under the lattice assumption as well as not,
a facet-to-facet conclusion for $k$-tilings is true in $\R^d$,
while for most values of $k$ and $d$ it is false.

There is a vast literature on the study of coverings of Euclidean space by a convex body,
and an equally vast body of work on classical tilings by translations of one convex body,
which must necessarily be a polytope (see for example \cite{Erdos61, Gritzmann}).
On the one hand, when we consider a $k$-tiling polytope $P$,
we obtain an exact covering of $\R^d$, in the sense that almost every point of $\R^d$
gets covered exactly $k$ times.   On the other hand, the family of $k$-tilers
is much larger than the family of $1$-tilers.
Hence the study of $k$-tilings lies somewhere between coverings and $1$-tilings.

It was known to Bolle~\cite{Bolle} that in $\R^2$,
every $k$-tiling convex polytope has to be a centrally symmetric
polygon, and  using combinatorial methods Bolle~\cite{Bolle} gave a
characterization for all polygons in $\R^2$ that admit a $k$-tiling with a {\it lattice} $\Lambda$
of translation vectors.
Kolountzakis~\cite{Kolountzakis04} proved that if a convex polygon $P$ tiles $\R^2$ multiply
with {\it any} discrete multiset $\Lambda$, then $\Lambda$ must
be  a {\it finite union of two-dimensional lattices}.
The ingredients of Kolountzakis' proof include the idempotent theorem for the Fourier transform of a measure.
Roughly speaking, the idempotent theorem of Meyer~\cite{meyer1970nombres}
tells us that if the square of the Fourier transform of a measure is itself,
then the support of the measure is contained in a finite union of lattices.
To put our main result into its proper context,
we record here the precise result of Kolountzakis.
A multiple tiling is called {\bf quasi-periodic} if its multiset
of discrete translation vectors $\Lambda$ is a finite union of translated lattices, not necessarily all of the same dimension.

\begin{theorem*}[Kolountzakis, 2002~\cite{kolountzakis2000structure}]
Suppose that $K$ is a symmetric convex polygon which is not a parallelogram.
Then $K$ admits only quasi-periodic multiple tilings if any.
\end{theorem*}

Here we extend this result to $\R^3$,
and we also find a fascinating class of polytopes analogous
to the parallelogram of the theorem above.
To describe this class, we first recall the definition of a {\bf zonotope},
which is the Minkowski sum of  a finite number of line segments.
In other words, a zonotope equals a translate of
$[-\bv_1, \bv_1] + \cdots + [-\bv_N, \bv_N]$, for some
positive integer $N$ and vectors $\bv_1,\ldots,\bv_N \in \R^d$.
A zonotope may equivalently be defined as  the projection of some $l$-dimensional cube.
A third equivalent condition is that for a $d$-dimensional zonotope, all of its
$k$-dimensional faces are centrally symmetric, for $1 \leq k \leq d$.
For example, the zonotopes in $\R^2$ are the centrally symmetric polygons.

We shall say that a polytope $P\subseteq\R^3$
is a  {\bf two-flat  zonotope} if $P$ is the Minkowski sum of $n+m$ line segments
which lie in the union of two different two-dimensional subspaces $H_1$ and $H_2$.
In other words, $H_1$ contains $n$ of the segments and $H_2$ contains $m$ of the
segments (if one of the segments belongs to both $H_1$ and $H_2$ we list it twice, once for each plane).
Equivalently, $P$ may be thought of as
the Minkowski sum of two $2$-dimensional symmetric polygons
one of which may degenerate into a single line segment.

Recently, a structure theorem for  convex $k$-tilers in $\R^d$ was found, and is as follows.
\begin{theorem*}[Gravin, Robins, Shiryaev 2011~\cite{gravin2011translational}]
If a convex polytope $k$-tiles $\R^d$ by translations,
then it is centrally symmetric and its facets are centrally symmetric.
\end{theorem*}
In the present context of $\R^3$,
it  follows  immediately from the latter theorem that a  $k$-tiler $P \subset \R^3$ is
necessarily a zonotope.
In this paper we extend the result of Kolountzakis~\cite{kolountzakis2000structure} from $\R^2$
to $\R^3$, providing a structure theorem for multiple tilings by polytopes in three dimensions.
\begin{main*}\label{main}
Suppose a polytope $P$ $k$-tiles $\R^3$ with a discrete multiset $\Lambda$, and suppose that
$P$ is not a two-flat zonotope.  Then $\Lambda$ is a finite union of translated lattices.
\end{main*}

 It turns out that if $P$ is a  rational two-flat zonotope, then $P$ admits a $k$-tiling with
a non-quasi-periodic set of translation vectors $\Lambda$,
as we show in our Corollary~\ref{cor:weird-tiling}.
For some of the classical study of $1$-tilings, and their interesting connections to zonotopes, the reader may refer to
the work of \cite{McMullen75}, \cite{McMullen80}, \cite{Shepard}, \cite{Ziegler_lectures}, and~\cite{Alexandrov}.
Here we find it very useful to use the intuitive language of
distributions \cite{rudin1973fa,strichartz2003guide} in order to think -- and indeed discover -- facts
about $k$-tilings.
To that end we introduce the distribution (which is locally a measure)
\beql{delta-lambda}
\delta_\Lambda:= \sum_{\lambda \in \Lambda}  \delta_\lambda,
\eeq
where $\delta_\lambda$ is the Dirac delta function at $\lambda \in \R^d$.
To develop some intuition, we may  check formally that
\[
\delta_\Lambda * \one_P = \sum_{\lambda \in \Lambda}  \delta_\lambda * \one_P = \sum_{\lambda \in \Lambda}  \one_{P+\lambda},
\]
so that from the first definition \eqref{k.tiling.def}
of $k$-tiling, we see that a polytope $P$ is a $k$-tiler if and only if
\begin{equation}\label{k.tiling.def2}
\delta_\Lambda * \one_P=k.
\end{equation}

The paper is modularized into short sections that highlight each step separately,
and the organization runs as follows.
In  \S\ref{sec:intro} we compute the Fourier transform of any $4$-legged frame of a polytope,
and show that its zeros form a certain countable union of hyperplanes.
In \S\ref{sec:condition} we find a sufficient condition,
which we call the {\em intersection property},
for the Fourier transform of ${\delta_\Lambda}$ to have a discrete support.  Then we
show that if $P$ is a $k$-tiler,
and if the intersection property holds for all $4$-legged frames of $P$,
then  $\supp\ft{\delta_\Lambda}$ (the support of the distribution
$\ft{\delta_\Lambda}$, the Fourier Transform of the distribution $\delta_\Lambda$)
is a discrete set in $\R^3$, of bounded density.

In \S\ref{sec:structure}, we prove that the intersection property implies the quasi-periodicity of $\Lambda$.

In \S\ref{sec:exceptions} we discover a fascinating family of $k$-tilers,
associated to a non-discrete $\supp{\ft{\delta_\Lambda}}$.
We prove that if $P$ tiles $\R^3$ with multiplicity, by translations with a discrete multiset $\Lambda$,
and the intersection property fails to hold, then $P$ must be a two-flat zonotope.

The proof of the Main Theorem is also given in \S\ref{sec:exceptions};
this proof is quite short since it just strings together all of the results of the previous sections.
In the final section, we  show that each rational two-flat zonotope
admits a very peculiar non-quasi-periodic $k$-tiling.  We note that it may be quite
difficult to offer a visualization of these $3$-dimensional non-quasi-periodic tilings,
and that we discovered them by using Fourier methods.

\section{Preliminaries}\label{sec:defs}

Suppose the polytope $P$ tiles multiply with the translates $\Lambda \subseteq \R^d$.
We will  need to understand some basic facts about how the
$\Lambda$ points are distributed, for example in Theorem~\ref{th:2d} below.

\begin{definition}
(Uniform density)\\
A multiset $\Lambda\subseteq\R^d$ has asymptotic density
$\rho$ if
$$
\lim_{R\to\infty} {\#(\Lambda \cap B_R(x)) \over \Abs{B_R(x)}} \to
\rho
$$
uniformly in $x\in\R^d$. In this case we write $\rho = \dens \Lambda$.
Another (weaker) notion is that of  {\it bounded}
density - we say that $\Lambda$ has (uniformly) bounded density if
$
 {\#(\Lambda \cap B_R(x)) \over \Abs{B_R(x)}} \le M
$
for $x\in\R^d$, and $R>1$. We say then that $\Lambda$ has density (uniformly) bounded by $M$.
\end{definition}

Since we intend to speak of the Fourier Transform of $\delta_\Lambda$ it is important to us
that if $\Lambda$ has bounded density then $\delta_\Lambda$
is a tempered distribution \cite{rudin1973fa,strichartz2003guide},
and, therefore, its Fourier Transform $\ft{\delta_\Lambda}$ is well defined.
And, it is almost obvious by comparing volumes that, if a polytope $P$ $k$-tiles $\R^d$ with
translates $\Lambda$ then $\Lambda$ has density $k/\Abs{P}$.

For any symmetric polytope $P$, and any face $F \subset P$,
we define $F^{-}$ to be the face of $P$ symmetric to $F$ with respect to
$P$'s center of symmetry.
We call $F^-$ the {\bf opposite face} of $F$.

Throughout the paper, we use the notation $\bx^\perp$ to denote the perpendicular subspace,
of codimension $1$, to the vector $\bx$.
We also use the standard convention of boldfacing all vectors,
to differentiate between $\bv$ and $v$, for example.
We furthermore use the convention that $[\be]$  denotes the
$1$-dimensional line segment from $0$ to the endpoint of the vector $\be$.
Whenever it is clear from context, we will also write $[e]$ to denote the same
line segment - for example, in the case that $e$ denotes an edge of a polytope.

We let $\bZ(f)$ be the zero set of the function $f$.

\begin{definition}[$4$-legged-frame of a polytope]
 \label{def:4legs}
 ~
\medskip

(a) Suppose $P\in\R^3$ is a zonotope (symmetric polytope with symmetric facets).
A collection of four (one-dimensional) edges of $P$ is called a  {\bf $4$-legged-frame} if whenever $e$ is one of the edges then
there exist two vectors $\btau_1$ and $\btau_2$ such that
the four edges are
$$
[e],\ [e]+\btau_1, [e]+\btau_2 \text{ and } [e]+\btau_1+\btau_2,
$$
and such that the edges
$[e]$ and $[e]+\btau_1$ belong to the same face of $P$ and the edges $[e]+\btau_2$ and $[e]+\btau_1+\btau_2$ belong to the opposite face.

\medskip

(b) For a set of four legs as above the {\bf leg measure} is the measure supported on the legs
and is equal to arc-length on the two legs $[e]$ and $[e]+\btau_1+\btau_2$ and minus arc-length on the
two legs $[e]+\btau_1$ and $[e]+\btau_2$. We denote this measure by $\mu_{e,\tau_1,\tau_2}$.
The leg measure is defined up to sign.
\end{definition}

\begin{remark}
A set of four legs of a symmetric polytope with symmetric faces is determined uniquely if we know
two opposite edges on a face (these edges are the first two legs). The other two legs are then the
corresponding opposite edges on the opposite face.
\end{remark}

\begin{figure}[h]
\label{fig:triangle}
\centering
\def\svgwidth{200pt}
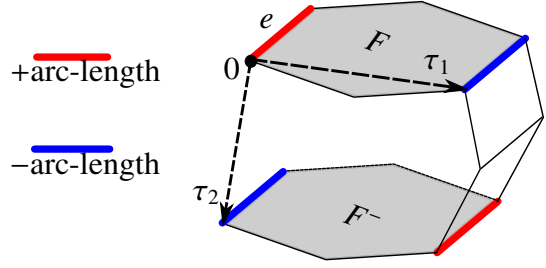
\caption{Four legs of a convex polytope and leg-measure.
The colored segments are edges of the polytope.
Translating an edge by $\pm\tau_1$ gives you the opposite edge on a face.
Translating by $\pm\tau_2$ takes you to the opposite face.}
\end{figure}

\section{The Fourier Transform of a $4$-legged frame}\label{sec:intro}

\begin{lemma}\label{lm:ft-zeroes-edges}
Suppose $\be, \btau_1, \btau_2 \in \R^3$ are linearly independent
and consider the measure  \\$\mu=\mu_{e,\tau_1,\tau_2}$ (see Definition \ref{def:4legs}).
Then the zero-set of the Fourier Transform $\hat{\mu}$, is
\beql{zero-set}
\bZ(\hat{\mu}) = \bbH_{_{-0}}(\be) \cup \bbH(\btau_1) \cup \bbH(\btau_2),
\eeq
where, if $\bx$ is a non-zero vector and $\Ds \bx^* = \frac{\bx}{\Abs{\bx}^2}$ is its geometric inverse,
\beql{def-s}
\bbH(\bx) = \Z \bx^* + \bx^\perp
\eeq
and
\beql{def-s-prime}
\bbH_{_{-0}}(\bx) = \left(\Z\setminus\Set{0}\right)\bx^* + \bx^\perp.
\eeq
Here $\bx^\perp$ is the hyperplane orthogonal to the vector $\bx$, so that
$\bbH(\bx) = \Z \bx^* + \bx^\perp$ is a collection of parallel hyperplanes, orthogonal to $\bx$
spaced by $1/\Abs{\bx}$.
\end{lemma}

\begin{proof}
Translating a measure does not alter the zero-set of its FT so we may translate $\mu$ so that $0$ is
the midpoint of the first line segment, which now runs from $-\be/2$ to $\be/2$.
Denoting by $\nu$ the arc-length measure on this line segment and writing
$\alpha=\delta_0-\delta_{\tau_1}$ and $\beta=\delta_0-\delta_{\tau_2}$ we obtain
$\mu$ as a convolution:
$$
\mu = \nu*\alpha*\beta.
$$
Taking the FT we get that
$$
\bZ(\hat{\mu}) = \bZ(\hat{\nu}) \cup \bZ(\hat{\alpha}) \cup \bZ(\hat{\beta}).
$$
Based on the calculation of the FT of the indicator function of $\left[-\frac{1}{2}, \frac{1}{2}\right]$
$$
\int_{-1/2}^{1/2} e^{-2\pi i \xi x}\,dx = \frac{\sin\pi\xi}{\pi\xi}
$$
we conclude that
$$
\ft{\nu}(\bu) = \Abs{\be}\frac{\sin\pi \langle \bu, \be\rangle}{\pi \langle \bu, \be\rangle}.
$$
One also immediately obtains the formulas
$$
\ft{\alpha}(\bu) = 2i e^{-\pi i \langle\btau_1, \bu\rangle} \sin{\pi \langle \btau_1, \bu\rangle}
$$
and
$$
\ft{\beta}(\bu) = 2i e^{-\pi i \langle\btau_2, \bu\rangle} \sin{\pi \langle\btau_2, \bu\rangle}.
$$
Since $\ft{\nu}$, $\ft{\alpha}$ and $\ft{\beta}$ vanish precisely on $\bbH_{_{-0}}(\be)$, $\bbH(\btau_1)$ and
$\bbH(\btau_2)$ respectively, the proof of Lemma \ref{lm:ft-zeroes-edges} is complete.
\end{proof}

\section{A sufficient condition for $\ft{\delta_\Lambda}$ to have discrete support}
\label{sec:condition}

\begin{theorem}\label{th:intersection}
Suppose $P$ is a symmetric polytope in $\R^3$ with symmetric faces and $\Lambda$ is a multiset
of points in $\R^3$ such that $P$ tiles at level $k$, a positive integer, when translated
at the locations $\lambda \in \Lambda$. Then we have
\beql{planes}
\supp\ft{\delta_\Lambda} \subseteq \Set{0} \cup
  \bigcap_{e, \tau_1, \tau_2} \left( \bbH_{_{-0}}(\be) \cup \bbH(\btau_1) \cup \bbH(\btau_2) \right),
\eeq
where $\delta_\Lambda$ is the measure corresponding to $\Lambda$ defined in \eqref{delta-lambda} and the intersection
in \eqref{planes} is taken over all $4$-legged frames $(e, \tau_1, \tau_2)$ of $P$.
\end{theorem}

\begin{proof}
We know from \cite{gravin2011translational} (see Lemma~$3.1$ and Lemma~$3.2$ in \cite{gravin2011translational})
that if $P$ tiles with $\Lambda$ and $\mu$ is a leg measure on $P$ then
$\mu$ also tiles with $\Lambda$, at level 0. In other words $\mu*\delta_\Lambda = 0$.
Since $P$ tiles when translated by $\Lambda$ it follows that $\Abs{\Lambda \cap [-R,R]^3} = O(R^3)$,
hence $\delta_\Lambda$ is a tempered distribution and we may take its FT which gives us
$\ft{\mu}\ft{\delta_\Lambda}=0$.
This implies (see the details in \cite[\S 1.2]{kolountzakis2004milano})
$$
\supp\ft{\delta_\Lambda} \subseteq \Set{0} \cup \bZ(\hat{\mu}).
$$
But the measure $\mu$ is exactly the one described in Lemma \ref{lm:ft-zeroes-edges}
and since this must be true for all sets of four legs of $P$ we conclude \eqref{planes}.
\end{proof}


\begin{corollary}\label{cor:intersection-property}
Suppose $P$ is a $k$-tiler with a discrete multiset $\Lambda$, in $\R^3$. Let the
following {\bf intersection property} hold:
\beql{intersection-property}
\bigcap_{e, \tau_1. \tau_2} \left( \be^\perp \cup \btau_1^\perp \cup \btau_2^\perp \right) = \Set{0},
\eeq
where the intersection above is taken over all sets of $4$-legged frames of $P$.\\
Then $\supp\ft{\delta_\Lambda}$ is a discrete set in $\R^3$, of bounded density.
\end{corollary}

\begin{proof}
The sets which are being intersected in \eqref{planes} are all unions of planes.
For this set to be non-discrete it must be the case that it contains an entire line  of direction,
say $\bu \in \R^3\setminus\Set{0}$.

This in turn implies that there is a selection $X_\ell$ of $\be$, $\btau_1$ or $\btau_2$ for
each set $\ell$ of four legs such that $\bu \in X_\ell^\perp$.
This contradicts condition \eqref{intersection-property}.

Having established that the intersection in the right hand side of \eqref{planes}
is a discrete point set we observe that the larger set
\beql{tmp4}
  \bigcap_{\be, \tau_1, \tau_2} \bbH(\be) \cup \bbH(\btau_1) \cup \bbH(\btau_2)
\eeq
is a finite union of discrete {\em groups}, each of them of the form
$$
\bigcap_{\ell} \bbH_{\ell},
$$
where $\ell$ runs through all possible sets of four legs of $P$ and for each $\ell=\Set{e,\tau_1,\tau_2}$ the set
$\bbH_{\ell}$ is one of $\bbH(\be), \bbH(\btau_1), \bbH(\btau_2)$.
Since each discrete group has bounded density so has the set \eqref{tmp4} and $\supp\ft{\delta_\Lambda}$
as its subset.
\end{proof}

\section{The intersection property  implies quasi-periodicity of $\Lambda$}
\label{sec:structure}

In this section we show how the discreteness of $\supp{\ft{\delta_\Lambda}}$ implies a rather rigid structure
for $\Lambda$.   Below we quote the result from \cite{kolountzakis2000structure},  where the multidimensional
statements had been proved in general, despite the fact that the final conclusions in
\cite{kolountzakis2000structure} are given only for dimension $2$.

\begin{theorem}[Kolountzakis, 2002]\label{th:2d}
 Suppose that for the multiset $\Lambda \subseteq \R^d$
 \begin{enumerate}
 \item\label{th:2d:1} $\Lambda$ has uniformly bounded density;
 \item\label{th:2d:2} $S:= \supp \widehat{\delta_{\Lambda}}$ is discrete;
 \item\label{th:2d:3} $ \left |S \cap  B_R(0)\right | \le  C\cdot R^d,$ for some positive constant $C$.
\end{enumerate}
Then $\Lambda$ is a finite union of translated $d$-dimensional lattices.
\end{theorem}

Next, we verify the conditions of the theorem above, for our $3$-dimensional $k$-tilers with a  multiset $\Lambda$.

\begin{claim}\label{cl:3d}
Suppose that convex polytope $P$ $k$-tiles $\R^{d}$ with $\Lambda$ and the intersection property
\eqref{intersection-property} of Corollary~\ref{cor:intersection-property} is true.
Then $\Lambda$ is a finite union of translated $d$-dimensional lattices.
\end{claim}
\begin{proof}
We just need to verify conditions \eqref{th:2d:1}, \eqref{th:2d:2} and \eqref{th:2d:3} of  theorem \ref{th:2d}.

Hypothesis~\eqref{th:2d:1} simply follows from the fact that in each sufficiently large ball $B_{R}(x)$ of $\R^{d}$ every point is covered exactly $k$ times by the translations of $P$ with the set $\Lambda\cap B_{R'}(x)$, where $R'=R+\diam P$.

Hypotheses~\eqref{th:2d:2} and \eqref{th:2d:3} follow from Corollary~\ref{cor:intersection-property}.
\end{proof}

\section{$k$-tilers associated to a non-discrete $\supp{\ft{\delta_\Lambda}}$}
\label{sec:exceptions}

In this section we study the convex polytopes that admit exceptional multiple tilings, in the sense
that the multiset of translations $\Lambda$ is not a finite union of $3$-dimensional lattices.  A class
 of these exceptions is easily provided by  prisms (Minkowski sums of a symmetric polygon with
a line segment, not in the polygons plane), for which one can lift a $2$-dimensional $k$-tiling
up into the third dimension by  separately $1$-tiling the tube above each projection.

By Claim~\ref{cl:3d} for such a tiling, the intersection property~\eqref{intersection-property}
cannot be true. Therefore, there exists a line (in fact a $1$-dimensional subspace) $l \subseteq \R^{3}$ such that
\beql{non-discrete-support}
 l   \subseteq   \bigcap_{\be, \tau_1. \tau_2} \left( \be^\perp \cup \btau_1^\perp \cup \btau_2^\perp \right) .
\eeq

It was already shown in \cite{gravin2011translational}
that a multiple tiler in $\R^3$ must be a zonotope,
i.~e. a Minkowski sum of line segments.
Here we will show that given the non-discreteness of $\supp{\ft{\delta_\Lambda}}$,
we can deduce that a zonotope is a Minkowski sum of two $2$-dimensional symmetric polygons.
And in Section~\ref{sec:weird} we provide
examples of such exceptional tilings, under a mild commensurability condition for such zonotopes.

\begin{theorem}\label{th:character}
Suppose  a polytope $P$ tiles $\R^3$ with multiplicity by translations over a multiset $\Lambda$
and condition \eqref{non-discrete-support}  holds. Then $P$ is a two-flat zonotope.
\end{theorem}

\begin{proof}

We let $L$ be a plane orthogonal to $l$ and supporting $P$; (\ref{non-discrete-support}) is then equivalent to
\beql{non-discrete-condition}
\forall \be, \btau_1, \btau_2, \text{ either } \be \parallel L, \text{ or } \btau_1 \parallel L,
  \text{ or } \btau_2 \parallel L.
\eeq

Let $F =  L\cap P$.
The dimension of the face $F$ can be $0$, $1$ or $2$.
Consider any facet $G$ of $P$ that has at least one common vertex with $F$, and let $e$ be an edge of $G$
that shares exactly one vertex $v$ with $F$ (so $G \neq F$).
Consider the $4$-legged frame determined by $G$ and $e$ with $\btau_1,\btau_2$
being the corresponding translation vectors.
Since $v$ is in $L$, by (\ref{non-discrete-condition}) one of the
three vertices $v+\be$, $v+\btau_1$ and $v+\btau_2$  lies in $L$, and therefore lies in $F$ as well.
By our choice of $\be$, $v+\be$ is a vertex of $G$ but not of $F$. Thus either $v+\btau_1\in F$,
or $v+\btau_2\in F$:
\begin{enumerate}
\item \label{case:tau1} If $v+\btau_1\in F$, then $\btau_1 \in G \cap F$, so we see that
$\btau_1$
is an edge of $G$.   Hence $G$ is a parallelogram.
\item \label{case:tau2}  If $v+\btau_2\in F$, then $F$ connects $G$ with its opposite face $G^-$.
\end{enumerate}
See Fig.\ \ref{fig:animals}.

\begin{figure}[h]
        \begin{subfigure}[b]{0.45\textwidth}
				        \label{fig:2flat_1}
								\centering
								\def\svgwidth{150pt}
								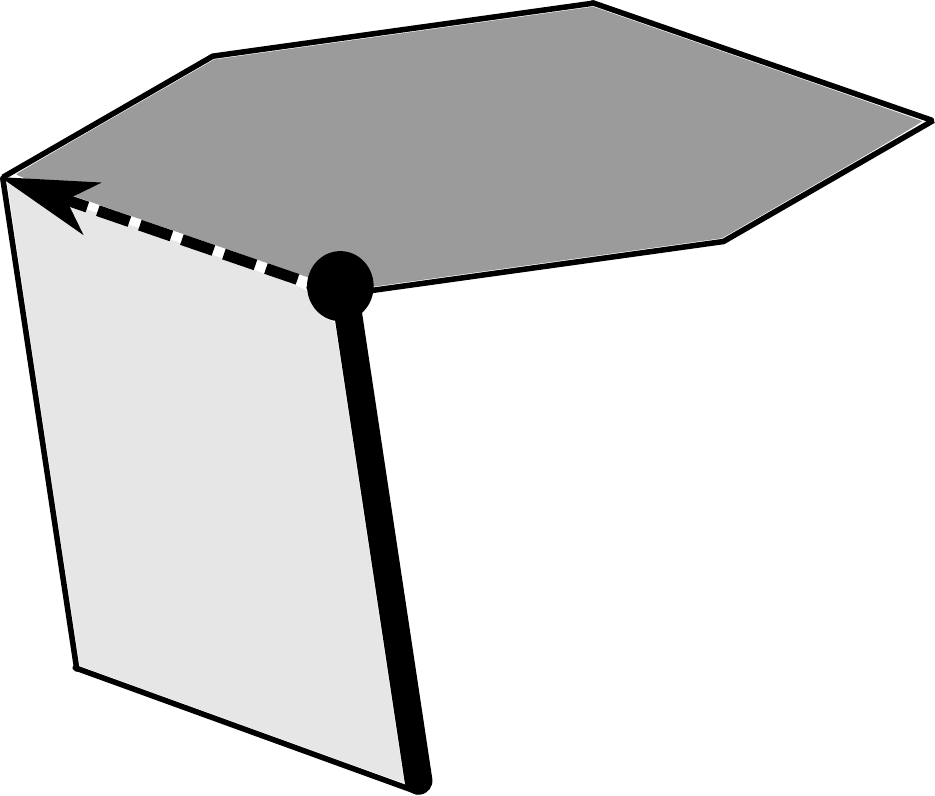
  				  		\caption{$v+\btau_1\in F.$}
        \end{subfigure}%
        \quad 
        \begin{subfigure}[b]{0.45\textwidth}
								\label{fig:2flat_2}
								\centering
								\def\svgwidth{150pt}
								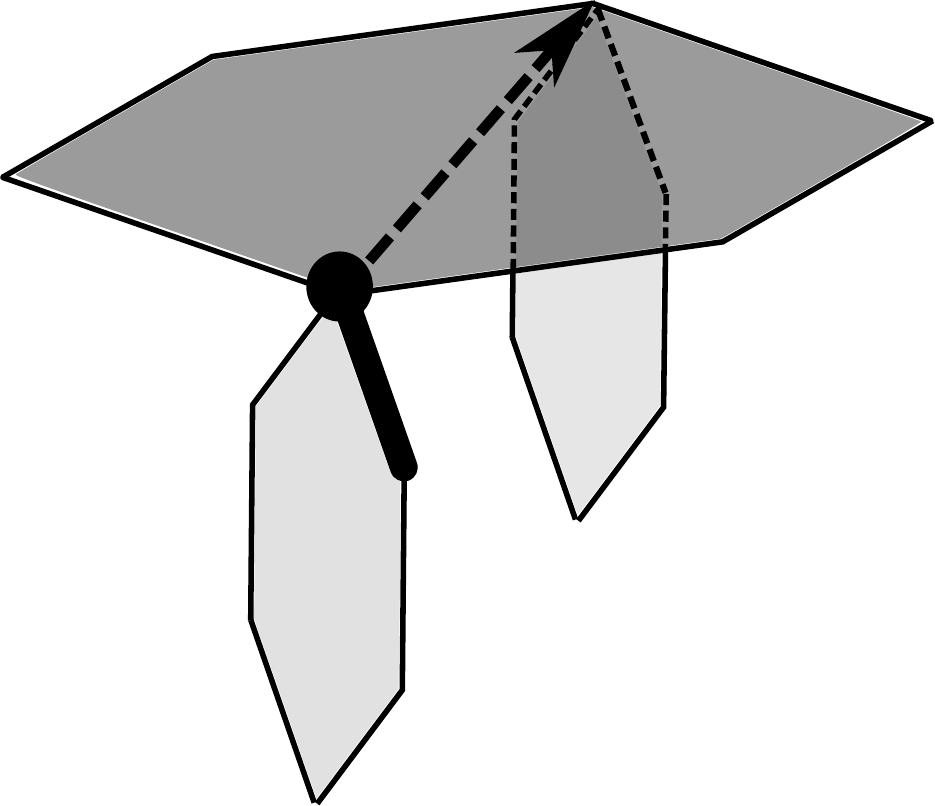
								\caption{$v+\btau_2\in F.$}
        \end{subfigure}
        ~ 
        \caption{The two possibilities for the facet $G$ with respect to $F$ }\label{fig:animals}
\end{figure}

It is our goal to find a facet $G$ which satisfies property \eqref{case:tau2}.
If, to the  contrary, every facet adjacent to $F$ satisfies property \eqref{case:tau1},
then each facet adjacent to $F$ is a parallelogram sharing an edge with $F$.
It follows that exactly three edges meet at every vertex of $F$, and all edges of these parallelograms
that are not edges of $F$ or parallel to $F$, are parallel to each other (see Fig.\ \ref{fig:prism}).
Now since $F$ is centrally symmetric,
consider two opposite parallel edges $e^+$ and $e^-$ of $F$, and corresponding parallelogram facets
$G$ and $G^-$. The facets $G$ and $G^-$ are parallel
and therefore opposite in $P$, therefore $G$  enjoys property \eqref{case:tau2}.
\begin{figure}[h]
\centering
\def\svgwidth{300pt}
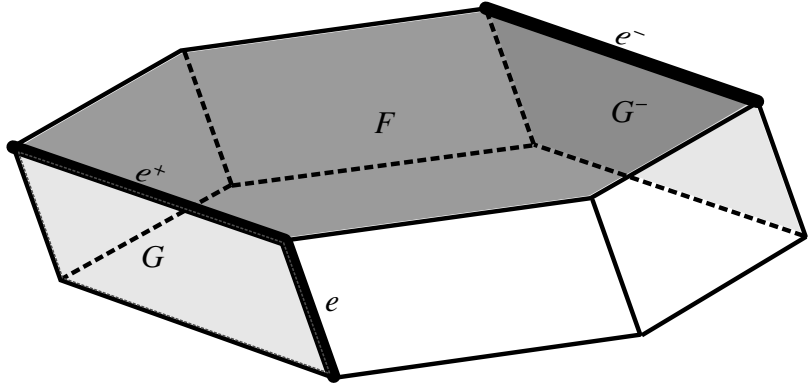
\caption{This is the case that each facet adjacent to $F$ is a parallelogram sharing an edge with $F$, giving us a prism}
\label{fig:prism}
\end{figure}

Now that we have found a facet $G$ such that $F$ connects $G$ with $G^-$,
we also note that since $P$ is centrally symmetric, $G$ also connects $F$ and $F^-$.
We will now show that $P=F+G$, under the minor assumption that $F$ and $G$ do not share an edge.  The case that $F$ and $G$ do in fact share an edge may be
handled in exactly the same manner, so without loss of generality
we assume throughout the rest of the proof that $F$ and $G$ share none of their edges.
Since $P$ is a zonotope, $P=F+G+H$ for some polytope $H$.
To arrive at a contradiction, we assume to the contrary that $H$ is not a single point,
and let $h_0$ be any edge of $H$.
Let $F=[f_1]+\dots+[f_k]$, $G=[g_1]+\dots+[g_\ell]$,
and $H=[h_0]+[h_1]+\dots+[h_m],$ where $k\ge 0$, $\ell\ge 2$, and $m\ge 0.$
We may assume that all line segments $f_i, g_i, h_i$ have the origin as their midpoint and thus
the center of $P$ is also at the origin.
We further consider a normal vector $\Bf_{\perp}$ to the face $F$ of $P$.
When $F$ is a $2$-dimensional face of $P$, $\Bf_{\perp}$
cannot be orthogonal to any line segment $h_i\in H$ and $g_i\in G$.
If $F$ is $0$ or $1$ dimensional face of $P$ (see figure~\ref{fig:FGframe}),
we have an infinite collection of perpendicular vectors to $F$ and we may choose
$\Bf_{\perp}$ to be not orthogonal to any line segment $h_i\in H$ and $g_i\in G$.
For each edge $g_i$ (resp. $h_i$) we define $\bg_i^+$ (resp. $\bh_i^+$) to be the vector from the origin to the endpoint of $g_i$ (resp. $h_i$) such that $\langle \bg_i^+ , \Bf_{\perp} \rangle>0$ (resp. $\langle \bh_i^+ , \Bf_{\perp} \rangle>0$). In the same way we define $\bg_i^-$ (resp. $\bh_i^-$) s.t.
$\langle \bg_i^- , \Bf_{\perp} \rangle<0$ (resp. $\langle \bh_i^- , \Bf_{\perp} \rangle<0$). Now one may easily see that
the location of the face $F$ in $\R^3$ is given by $[f_1]+\dots+[f_k]+\bg_1^{+}+\dots+\bg_\ell^{+} + \bh_0^{+}+\dots+\bh_m^{+}$ as a set of extremal points of the linear functional corresponding to $\Bf_{\perp}$. Similarly
the location of the face $F^-$ in $\R^3$ is given by $[f_1]+\dots+[f_k]+\bg_1^{-}+\dots+\bg_\ell^{-} + \bh_0^{-}+\dots+\bh_m^{-}.$ Therefore the distance between $F$ and $F^-$ is
\[
\langle
\Bf_{\perp},  \sum_{i=1}^{\ell}\bg_i^{+} + \sum_{i=0}^{m}\bh_i^{+} -\sum_{i=1}^{\ell}\bg_i^{-} - \sum_{i=0}^{m}\bh_i^{-}
\rangle > \langle
\Bf_{\perp},  \sum_{i=1}^{\ell}\left(\bg_i^{+} -\bg_i^{-}\right)
\rangle.
\]

\begin{figure}[h]
\centering
\def\svgwidth{300pt}
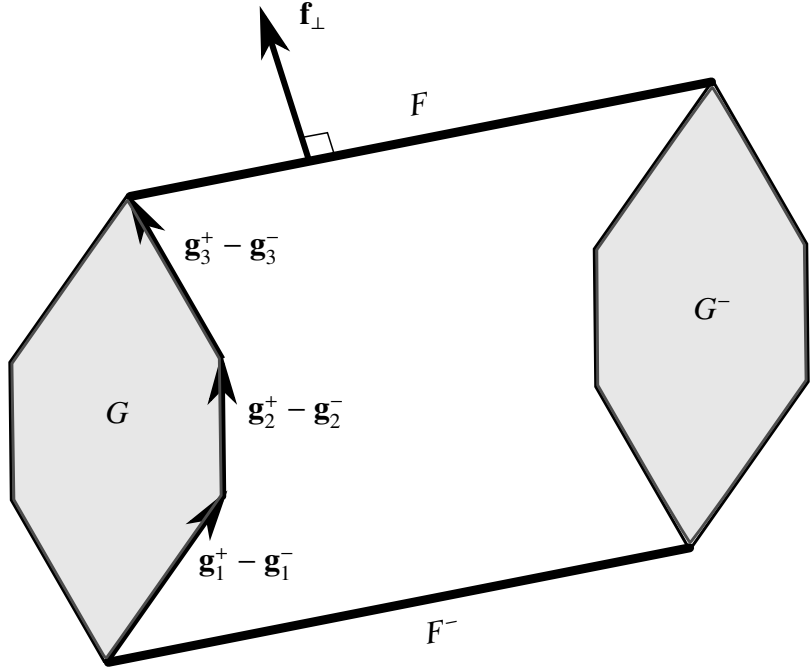
\caption{Here $F$ is a lower-dimensional face of $P$, namely an edge of $P$, and we see how we can get from the face $F$ to the face $F^-$ by walking along the vectors  $\bg_i^+ -   \bg_i^-$.   Here $\Bf_\perp$ is chosen to be a vector orthogonal to $F$ and not orthogonal to any of the line segments $h_j$.
}
\label{fig:FGframe}
\end{figure}

On the other hand, since $G$ connects $F$ and $F^-$ we have $F=F^-+\sum_{i\in I}\left(\bg_i^+-\bg_i^-\right),$ for a set $I$ of edges in $G$. The latter implies that the distance between $F$ and $F^-$ is not more than

\[
\langle
\Bf_{\perp},  \sum_{i=1}^{\ell}\left(\bg_i^{+} -\bg_i^{-}\right)
\rangle,
\]
a contradiction.

\end{proof}

\begin{remark}
One of $F$ or $G$ can be $1$-dimensional, in which case $P$ becomes a $3$-dimensional prism.
\end{remark}

\begin{main*}
Suppose a polytope $P$ $k$-tiles $\R^3$ with a multiset $\Lambda$, and suppose that
$P$ is not a two-flat zonotope.  Then $\Lambda$ is a finite union of translated lattices.
\end{main*}
\begin{proof}
If $P$ is not a two-flat zonotope, then due to  Theorem~\ref{th:character},
 condition \eqref{non-discrete-support} is violated.    Therefore, the intersection property
\eqref{intersection-property} in Corollary~\ref{cor:intersection-property} holds.
The Claim~\ref{cl:3d} now concludes the proof.
\end{proof}

\section{Many two-flat zonotopes have weird tiling sets}\label{sec:weird}

In this section we prove that, under a mild commensurability condition, two-flat zonotopes
admit tilings which are not quasi-periodic (``weird'').

\begin{theorem}\label{weirdtiling}
Suppose $P$ is a two-flat zonotope in $\R^3$ which is the Minkowski sum of the segments
$$
[\bv_1], \dots, [\bv_n], [\bw_1], \dots, [\bw_m],
$$
where $\bv_1, \ldots, \bv_n \in H_1$ and $\bw_1,\ldots,\bw_m \in H_2$
and $H_1$, $H_2$ are two different two dimensional subspaces.
Suppose also that the additive group generated by $\bv_1,\ldots,\bv_n,\bw_1,\ldots,\bw_m$ is discrete and that the
$\bv_j$ span $H_1$.

Then $P$ admits a tiling by translations at a set $\Lambda\subseteq\R^3$ which
is not a finite union of translated lattices.
\end{theorem}

\begin{corollary}\label{cor:weird-tiling}
If $P \subseteq \R^3$ is a two-flat {\em rational}
zonotope then $P$ admits tilings
by sets which are not finite unions of translated lattices.
\end{corollary}

\begin{proof}[Proof of Theorem \ref{weirdtiling}]
We begin the analysis by noting that $P$
can be paved by  parallelepipeds, whose sides
are among the vectors $\bv_j$ and $\bw_j$ (proof is by induction on
the number of line segments whose Minkowski sum is the zonotope).
Therefore we can write its indicator function as a finite sum of
indicator functions of parallelepipeds.
\[
{\one}_P(x) = \sum_{j=1}^M  {\one}_{B_j}(x),\ \ \mbox{for a.e.\ $x$},
\]
where each $B_j$ is a parallelepiped, whose three sides are equal to some
of the $\bv_j$ and $\bw_j$.

Suppose now that the parallelepiped $B$ is centered at the origin and has sides parallel to the
three linearly-independent vectors $\ba, \bb, \bc$. We can write the indicator function of $B$
as a convolution
$$
\one_{B} = \frac{\Abs{\det(\ba,\bb,\bc)}}{\Abs{\ba}\cdot\Abs{\bb}\cdot\Abs{\bc}} \mu_{\ba}*\mu_{\bb}*\mu_{\bc},
$$
where $\mu_{\ba}$ is the measure that equals arc-length on the line segment from $-\ba/2$ to $\ba/2$,
and $\mu_{\bb}, \mu_{\bc}$ are similarly defined.
Since (see \S\ref{sec:intro})
$$
\ft{\mu_{\ba}}(\bxi) = \Abs{\ba} \frac{\sin{\pi\inner{\bxi}{\ba}}}{\pi\inner{\bxi}{\ba}}
$$
and similarly for $\ft{\mu_{\bb}}$, $\ft{\mu_{\bc}}$, we obtain the formula
\beql{ppft}
\ft{\one_B}(\bxi) = \Abs{\det(\ba,\bb,\bc)} \frac{\sin{\pi\inner{\bxi}{\ba}}}{\pi\inner{\bxi}{\ba}} \cdot
     \frac{\sin{\pi\inner{\bxi}{\bb}}}{\pi\inner{\bxi}{\bb}} \cdot
     \frac{\sin{\pi\inner{\bxi}{\bc}}}{\pi\inner{\bxi}{\bc}}.
\eeq
Each parallelepiped $B_j$ in the decomposition of $P$ is a translate of a parallelepiped of
the type $B$, above, with some of the vectors $\bv_j$, $\bw_j$ in place of $\ba, \bb, \bc$.
Hence the Fourier Transform $\ft{\one_{B_j}}$ has the same zeros as the Fourier Transform
of its centered translate and these are
$$
\fzero{\one_{B_j}} = \left((\Z')\ba^*+\ba^\perp\right) \cup
	\left((\Z')\bb^*+\bb^\perp\right) \cup \left((\Z')\bc^*+\bc^\perp\right),
$$
where $\Z'=\Z\setminus\Set{0}$ and again $\ba^*=\ba/\Abs{\ba}^2$ is the geometric inverse of $\ba$, etc.
Write now
$$
G = \langle \bv_1, \ldots, \bv_n \rangle
$$
for the additive subgroup (lattice) of $H_1$ generated by the $\bv_j$'s and $G^* \subseteq H_1$
for its dual lattice {\em in $H_1$}
\beql{tmp2}
G^* = \Set{\bu\in H_1:  \forall \bg \in G \quad \inner{\bu}{\bg} \in \Z}.
\eeq
We claim now that for each $j$
\beql{ppipedzeros}
H_1^\perp + \left( G^*\setminus(\bv_1^\perp \cup \ldots \cup \bv_n^\perp) \right)  \ \subseteq\ \fzero{\one_{B_j}}.
\eeq
This follows since at least one side of $B_j$ is equal to a vector $\bv_j$
which makes the corresponding factor in \eqref{ppft} vanish on any element of $G^*$ which is not orthogonal to
$\bv_j$.
And since that factor in $\eqref{ppft}$ is constant along $H_1^\perp$ we obtain the claim.
Since \eqref{ppipedzeros} holds for all $j$ we obtain
\beql{pfzeros}
H_1^\perp + \left( G^*\setminus(\bv_1^\perp \cup \ldots \cup \bv_n^\perp) \right) \ \subseteq\ \fzero{\one_{P}}.
\eeq

Pick now any non-zero $c_1, c_2, \ldots, c_n \in \R$. We claim that the measure
\begin{equation}
\label{eq:zero}
\tau := \one_P *\delta_G * \left[\delta_0 - \delta_{c_1 \bv_1}\right] * \dots
               *  \left[\delta_0 - \delta_{c_n \bv_n}\right] = 0,
\end{equation}
where $\delta_G = \sum_{\bg\in G} \delta_{\bg}$.
For this it is enough to show that the Fourier Transform of the above measure
$$
\ft{\tau}(\bxi) = \ft{\one_P}(\bxi) (1-e^{2\pi i c_1\inner{\bv_1}{\bxi}})\cdots
	(1-e^{2\pi i c_n\inner{\bv_n}{\bxi}}) \ft{\delta_G}
$$
is identically $0$.
By the Poisson Summation Formula (the Fourier Transform is taken in the sense of
distributions)
\beql{poisson}
\ft{\delta_L} = \frac{1}{\vol{L}} \delta_{L^*},
\eeq
for each lattice $L=A\Z^d$ in $\R^d$ and dual lattice $L^*=A^{-\top}\R^d$ (where $A \in GL(d,\R)$),
it follows that $\ft{\delta_G}$ is a {\em measure}
with support on the lines orthogonal to $H_1$ that
go through the points of $G^*$:
$$
\supp{\ft{\delta_G}} = G^* + H_1^\perp.
$$
By \eqref{pfzeros} the function $\ft{\one_P}(\bxi)$ kills $\ft{\delta_G}$ except at the
lines of the form $\bg^*+H_1^\perp$ with $\bg^* \in G^*$
is orthogonal to some $\bv_j$. But at these lines one of the factors
$$
(1-e^{2\pi i c_1\inner{\bv_1}{\bxi}})\cdots
        (1-e^{2\pi i c_n\inner{\bv_n}{\bxi}})
$$
vanishes, so indeed $\ft{\tau}$ is zero.
Now rewrite the measure $(\delta_0-\delta_{c_1\bv_1})\cdots(\delta_0-\delta_{c_n\bv_n})$ in the form
$$
\sum_{k=1}^N \delta_{\bu_k^+} - \sum_{k=1}^N \delta_{\bu_k^-},\ \ \ (\mbox{where $N=2^{n-1}$}).
$$
Equivalently, we can rewrite \eqref{eq:zero} as the equality
\begin{equation}
\label{eq:slab_equality}
\one_P * \delta_G * \sum_{k=1}^N \delta_{\bu_k^+} = \one_P * \delta_G * \sum_{k=1}^N \delta_{\bu_k^-}.
\end{equation}
Define the multisets
\beql{st}
S = G+\Set{\bu_1^+,\ldots,\bu_N^+} \mbox{ and } T = G+\Set{\bu_1^-,\ldots,\bu_N^-}
\eeq
whose ground sets are the supports of the discrete measures
$$
\delta_G * \sum_{k=1}^N \delta_{\bu_k^+} \mbox{ and }
\delta_G * \sum_{k=1}^N \delta_{\bu_k^-},
$$
and their multiplicities at each point are those described by these measures.

In what follows we exploit \eqref{eq:slab_equality}
to give an example of a multiple tiling by $P$ with a {\bf discrete
set $\Lambda$}, which by no means can be expressed as a finite union of
translated lattices.

We notice first that since $P$ is a zonotope decomposing into parallelepipeds
of sides among the vectors $\bv_j, \bw_j$, it $k$-tiles $\R^3$, for some $k$, with the lattice
$$
\Gamma = \langle \bv_1, \ldots, \bv_n, \bw_1, \ldots, \bw_m \rangle
$$
generated by the $\bv_j, \bw_j$.
The reason is that each of the parallelepipeds $B_j$ tiles with a subgroup of $\Gamma$ (the group
generated by its side vectors) and therefore it tiles multiply with $\Gamma$ itself.
Clearly, $P$ also $(Nk)$-tiles $\R^3$ by the union of
$N$ translations of the lattice $\Gamma$ by the vectors $\bu_1^+,\dots,\bu_N^+$.

Let $\Set{\gamma_j:\ j\in\Z}$ be a complete set of coset representatives of $G$ in $\Gamma$.
Define the set of translates
$$
\Lambda = \bigcup_{j \in \Z} ( E_j + \gamma_j ),
$$
where for each $j\in\Z$ we choose $E_j = S$ or $E_j = T$ arbitrarily.
We claim that for any such choice of the $E_j$ the $\Lambda$-translates of $P$ form a $(Nk)$-tiling of $\R^3$.
Indeed the claim is true if all $E_j=S$ as it is a restatement of the fact that $P$ $(Nk)$-tiles with
$\Gamma+\Set{\bu_1^+,\ldots,\bu_N^+}$,
which itself follows from the fact that $P$ $k$-tiles with $\Gamma$.
Observe now that if we change any single $E_j$ from $S$ to $T$ we are adding the quantity
\beql{correction}
\one_P * \delta_G * \sum_{i=1}^N\delta_{\bu_i^-} * \delta_{\gamma_j} -
     \one_P * \delta_G * \sum_{i=1}^N\delta_{\bu_i^+} * \delta_{\gamma_j}
\eeq
to the constant function
$$
\one_P*\delta_{\Lambda},
$$
which therefore remains the same since \eqref{correction} is identically $0$ by \eqref{eq:slab_equality}.
We conclude that we have a $(Nk)$-tiling no matter how the $E_j$ are chosen
(one has to make the remark here that in any given bounded region of space
the fact that $P+\Lambda$ is a $(Nk)$-tiling
or not is affected by finitely many choices for the $E_j$).

Choose now all $E_j=S$ with the exception of $E_0=T$.
We claim that the corresponding set $\Lambda$ is not a finite union of translated fully-dimensional lattices.
Indeed, by the Poisson Summation Formula \eqref{poisson} we have that,
if

$$
\Lambda' = \bigcup_{j\in\Z} \left( S+\gamma_j \right) = \Gamma + \Set{\bu_1^+,\dots,\bu_N^+}
$$
then $\ft{\delta_{\Lambda'}}$
is a discrete measure in $\R^3$ and this should also be true for $\ft{\delta_\Lambda}$ if $\Lambda$
too were a finite union of translated lattices.
Thus the difference
$$
\ft{\delta_{\Lambda'}} - \ft{\delta_\Lambda}
$$
would also be a discrete measure.
But
\begin{align*}
\delta_{\Lambda'}-\delta_{\Lambda} & = \delta_{S+\gamma_0} - \delta_{T+\gamma_0}\\
 & = \delta_{\gamma_0} * \delta_G * \sum_{i=1}^N \left( \delta_{\bu_i^+}-\delta_{\bu_i^-} \right) \\
 & = \delta_{\gamma_0} * \delta_G * (\delta_0-\delta_{c_1\bv_1})*\cdots * (\delta_0 - \delta_{c_n\bv_n}).
\end{align*}
so
\beql{tmp3}
\ft{\delta_{\Lambda'}} - \ft{\delta_{\Lambda}} =
  e^{2\pi i \inner{\gamma_0}{\bxi}} \prod_{j=1}^n (1-e^{2\pi i c_j\inner{\bv_j}{\bxi}}) \ft{\delta_G}.
\eeq
Recall now that the support of the measure $\ft{\delta_G}$ are all straight lines orthogonal to $H_1$
passing through a point of $G^*$, the dual lattice to $G$ in $H_1$.
The factors in the right hand side of \eqref{tmp3} vanish at the set
\beql{parallel-planes}
\bigcup_{j=1}^n \left( \Z \frac{\bv_j^*}{c_j} + \bv_j^\perp \right).
\eeq
Each set in this union consists of a series of planes normal to $\bv_j$ and spaced by a length of
$\left(c_j \Abs{\bv_j}\right)^{-1}$.
Each of the straight lines that make up the support of $\ft{\delta_G}$ is parallel to each such plane and,
therefore, each such line is either entirely contained in \eqref{parallel-planes} or is disjoint from it.
It follows that, since the right-hand side of \eqref{tmp3} is not identically zero, its support
contains at least one straight line of the direction $H_1^\perp$ and is not a discrete set, as we had to show.
\end{proof}

\begin{remark} One may easily extend the previous construction of $\Lambda$ to the examples that cannot be expressed as a linear combination of finitely many
possibly lower-dimensional lattices.
\end{remark}
\begin{proof}
In the previous construction of $\Lambda$ we could let $E_j$ to be either $S$, or $T$ for each $j$ and still get a legitimate $(Nk)$-tiling of $\R^3$.
In general we could have a big family (of cardinality $2^{\Z}$) of possible $Nk$-tilings of $\R^3$.
We call a tiling {\bf weird} if it is not quasi-periodic. In what follows we show that our big family
has a weird member $\Lambda^{\dag}$. In our construction we will need the following claim.
\begin{claim}\label{coloring}
The set of integers $\Z$ can be colored with two colors in such a way that every arithmetic progression has
infinitely many numbers of each color.
\end{claim}
\begin{proof}
There are countably many arithmetic progressions in $\Z$. We enumerate them all denoting $A_i$ the the $i$'th progression in the enumeration,
such that any progression appears infinitely many times. We begin to color $\Z$ step by step in such a way, that on the $n$'th step all progressions $A_i$ for $i$ from $1$ to $n$ have numbers of both colors.
At the $n$'th step we find two numbers of $A_n$ that are not yet colored, and color them differently. It is always possible to do so, because at step $n$ only finitely many numbers of $\Z$ are already colored, and $A_n$ has infinitely many numbers. With such a coloring every arithmetic progression would contain infinitely many integers of each of the colors.
\end{proof}

We can pick our complete set of the coset representatives $\Set{\gamma_j:\ j\in\Z}$ so that it contains $\gamma_1\cdot\Z$  as a subset.

In order to construct $\Lambda^{\dag}$, we consider a coloring of $\gamma_1\cdot\Z$ with two colors (red and black) so that every arithmetic progression there has infinitely many red and infinitely many black members.
We let $E_j=S$ if corresponding coset representative $\gamma_j\notin \gamma_1\cdot\Z$. For coset representatives in $\gamma_1\cdot\Z$, if the point $\gamma_1\cdot j$ is red we choose $E_j=S$, if the point $\gamma_1\cdot j$ is black we choose $E_j=T$ in $\Lambda^{\dag}$.

We further notice that due to the freedom to choose $c_k$'s in the definition of
$\bu^{-}_{k}$'s and $\bu^{+}_{k}$'s, one can pick $c_k$'s so that multisets $S$ and $T$
have different multiplicities at $0$. Indeed, we may pick $c_k$'s so that for any set of indexes $I\subset[n]$ the corresponding linear combination $\sum_{k\in I} c_k\cdot \bv_k\notin G$. Then $G+\sum_{k=1}^{N}\bu^{+}_{k}$ has $0$ at multiplicity $1$, while $G+\sum_{k=1}^{N}\bu^{-}_{k}$ has $0$ at multiplicity $0$.
Furthermore, at each point $\gamma_1\cdot\ell$ of $\gamma_1\cdot\Z$ the multisets
$S+\ell\cdot\gamma_1$ and $T+\ell\cdot\gamma_1$ have different multiplicities as well.
Thus we get irregular behavior of $\Lambda^{\dag}$ on the line $\Z\cdot\gamma_1$.
In particular, $\Lambda^{\dag}$ simultaneously contains and misses infinitely many members of each infinite coset of $\Z\cdot\gamma_1$.

Now if we assume that $\Lambda^{\dag}$ may be expressed as a finite linear combination of translated lattices $\delta_{_{\Lambda^{\dag}}}=q_1\cdot\delta_{_{\Lambda_1}} + \ldots +q_m\cdot\delta_{_{\Lambda_m}}$ (to simplify notations, we will write $q_1\cdot\Lambda_1 + \ldots +q_m\cdot\Lambda_m$ instead of
$\delta_{_{\Lambda^{\dag}}}$), then
$$
\Lambda^{\dag}\cap\Z\cdot\gamma_1  = \sum_{i=1}^m q_i\cdot\left(\Lambda_i\cap\Z\cdot\gamma_1\right).
$$

Each $\Lambda_i\cap\Z\cdot\gamma_1$ is a coset of $\Z\cdot\gamma_1$. Therefore, $\Lambda_i\cap\Z\cdot\gamma_1$ is either empty, or
is a single point, or is an arithmetic progression in $\Z\cdot\gamma_1$ with the common difference $d_i$. We denote the set of all the indices of the
latter $\Lambda_i$'s by $M\subset\{1,\dots,m\}$.
We further consider an arithmetic progression $A$ of $\Z\cdot\gamma_1$ with the common
difference $\prod_{i\in M} d_i$. We notice that for any $i\in M$ either $A\cap\Lambda_i=A$, or $A\cap\Lambda_i=\emptyset$.
Since $A\subset \Z\cdot\gamma_1$, we have

$$
\Lambda^{\dag}\cap A = \sum_{i=1}^m q_i\cdot\left(\Lambda_i\cap A\right)=\sum_{i\notin M}q_i\cdot\left(\Lambda_i\cap A\right)+
A\cdot\sum_{\substack{i\in M:\\ \Lambda_i\cap A\neq\emptyset}}q_i.
$$

According to the definition of $M$ the set $\sum_{i\notin M}q_i\cdot\delta_{_{\Lambda_i\cap A}}$ has finite support.
Since $A$ is an arithmetic progression in $\Z\cdot\gamma_1$ and due to our construction of $\Lambda^{\dag}$, the support of $$\delta_{_{\Lambda^{\dag}\cap A}} -
\delta_{_A}\cdot\sum_{\substack{i\in M:\\ \Lambda_i\cap A\neq\emptyset}}q_i
$$
cannot be finite, a contradiction.
\end{proof}

\bibliographystyle{abbrv}
\bibliography{Bibliography_Short}

\end{document}